\documentclass[ a4paper, english]{amsart}
\usepackage{mathptmx}

\usepackage[T1]{fontenc}
\usepackage[latin1]{inputenc}
\usepackage{amsthm}
\usepackage{graphicx}
\usepackage{bbm}
\usepackage{mathptmx}
\usepackage{setspace}
\usepackage{color}
\usepackage{cancel}
\usepackage{a4wide}
\def\bewend{\hspace{\stretch{1}}$\square$}
\makeatletter

\usepackage{babel}

\numberwithin{equation}{section} 
\numberwithin{figure}{section} 
\theoremstyle{plain}
\newtheorem{thm}{Theorem}
  \theoremstyle{plain}
  \newtheorem{cor}{Corollary}[section]
  \theoremstyle{plain}
  \newtheorem{lem}[cor]{Lemma}
  \theoremstyle{plain}
    \newtheorem{prop}[cor]{Proposition}
    \theoremstyle{remark}
       \newtheorem{rem}[cor]{Remark}
	 \theoremstyle{plain}

\theoremstyle{plain}
  
\usepackage{bbm}

\renewcommand{\phi}{\varphi}

\usepackage{babel}
\usepackage[unicode=true,
 breaklinks=false,
 pdfborder={0 0 1},
 backref=false,
 colorlinks=false]
 {hyperref}
\hypersetup{%
 pdfauthor={Marc Kesseboehmer, Sara Munday, and  Bernd O. Startmann} }

\def\R{\mathbb{R}}

\def\N{\mathbb{N}}

\def\U{\mathcal{U}}

\def\var{\epsilon}

\usepackage{mathptmx}

\begin{document}
\onehalfspacing

\title
{A note on Diophantine fractals for $\alpha$--L\"uroth systems.}

\author{Sara Munday}

\address{Mathematical Institute, University of St. Andrews, North Haugh, St.
Andrews KY16 9SS, Scotland}

\email{sam20@st-andrews.ac.uk}

\date{\today}
\maketitle

\section{Introduction}

In this paper, we give the Hausdorff dimensions of certain sets of real numbers described in terms of  the $\alpha$-L\"uroth expansion. So, let us first describe this expansion.

Let $\alpha:=\{A_n:n\in\N\}$ denote a countable partition of the unit interval $\U$, consisting of right-closed, left-open intervals which we always assume to be ordered from right to left, starting from $A_1$. Let $a_n$ denote the Lebesgue measure $\lambda(A_n)$ of the
atom $A_n\in\alpha$ and let $t_n:=\sum_{k=n}^\infty a_k$
denote the Lebesgue measure of the $n$-th tail of $\alpha$.
Then, for a given partition $\alpha$, define the map $L_\alpha:\U\to\U$ by setting
\[
L_{\alpha}(x):=
\left\{
	\begin{array}{ll}
	    ({t_n-x})/a_n & \text{ for }x\in A_n,\ n\in\N,\\
	  0 & \hbox{ if } x=0.
	\end{array}
      \right.
\]
The map $L_\alpha$ is referred to as the \index{$\alpha$-L\"uroth map}\textit{$\alpha$-L\"uroth map}.

for each partition $\alpha$ the map $L_{\alpha}$ gives rise to a series expansion of
numbers in
the interval $\mathcal{U}$, which we refer to as the
\index{expansion!$\alpha$-L\"uroth}\textit{$\alpha$-L\"{u}roth expansion}.
That is,
let $x\in\mathcal{U}$ be given and let the finite or infinite sequence $(\ell_k)_{k\geq1}$ be determined by $L_{\alpha}^{k-1}(x) \in
A_{\ell_{k}}$. Note that the sequence will be finite if at some point we have that $L_{\alpha}^{k}(x)=0$ and also note that each finite sequence has the property that the final entry is at least equal to 2. This sequence gives rise
to an alternating series expansion of each $x\in\U$, which is given by
\[
x=t_{\ell_1}+
\sum_{n=2}^\infty(-1)^{n-1}\left(\textstyle\prod\limits_{i<n}a_{\ell_i}\right)
t_{\ell_n}=t_{\ell_1}-a_{\ell_1}t_{\ell_2}+a_{\ell_1}a_{\ell_2}t_{\ell_3}-\ldots
\]
Let us denote finite $\alpha$-L\"uroth expansions by $[\ell_1, \ell_2, \ldots, \ell_k]_\alpha$, for some $k\in\N$, and  infinite ones by
$x=[ \ell_1, \ell_2,
\ell_3, \ldots]_{\alpha}$. Each infinite expansion is unique.

\begin{rem}
Note that this series expansion is a particular type of \textit{generalised L\"uroth series}, a concept which was introduced by Barrionuevo \textit{et al.} in \cite{BBDK} (Also see the book by Dajani and Kraaikamp, \cite{DK}).
\end{rem}

Throughout this paper, we will make the additional assumption that the tails of the partition $\alpha$ satisfy the power law $t_n=\psi(n)\cdot n^{-\theta}$, where $\psi:\N\to \R^+$ is a slowly-varying function\footnote{A  measurable function $f:\R^{+} \to \R^{+}$ is said
to be \textit{slowly-varying}  if $
\lim_{x\to\infty}f(x y)/f(x)=1$, for all $y>0$.}. Such a partition is said to be \index{expansive partition}\textit{expansive of exponent $\theta\geq0$}. Also, we always assume that every partition $\alpha$ is \textit{eventually decreasing}, that is, that for all sufficiently large $n\in\N$, we have that $a_n<a_{n-1}$. The following Proposition appears in \cite{KMS}.

\begin{prop}\label{mdt}
If $\alpha$  is  expansive of exponent $\theta>0$ and eventually decreasing, then we
    have that
    \[
    a_{n}\sim n^{-1} t_n.\]
\end{prop}

Our first main theorem concerns $\alpha$-Good sets, which are defined as follows. For each $N\in\N$, let the set $G_N^{(\alpha)}$ be defined by
\[
G_N^{(\alpha)}:=\{x=[\ell_1(x), \ell_2(x), \ldots]_\alpha\in\U:\ell_i(x)> N\text{ for all }i\in\N\}.
\]
Note that the name ``Good'' here refers to I.J. Good \cite{Good}, for the similar results he proved for continued fractions, and not to any supposed nice property of these sets.
We have the following result.
\begin{thm}\label{origgood}

\[
\lim_{N\to \infty}\dim_H\left(G_N^{(\alpha)}\right)= \frac{1}{1+\theta}.
\]
\end{thm}

For our second main result, let us consider the following sets. Define
\[
F_{\infty}^{(\alpha)}:=\left\{x=[\ell_1(x), \ell_2(x), \ldots]_\alpha:\lim_{n\to\infty}\ell_n(x)=\infty\text{\ and }\ell_n(x)\geq\ell_{n-1}(x)\right\}
\]
and
\[
G_{\infty}^{(\alpha)}:=\left\{x=[\ell_1(x), \ell_2(x), \ldots]_\alpha:\lim_{n\to\infty}\ell_n(x)=\infty\right\}.
\]
It is immediately apparent that $F_{\infty}^{(\alpha)}\subset G_{\infty}^{(\alpha)}$, so that $\dim_H\left(F_{\infty}^{(\alpha)}\right)\leq \dim_H\left(G_{\infty}^{(\alpha)}\right)$. We aim to prove the following theorem.

\begin{thm}\label{goodtwo}
\[\dim_H\left(F_\infty^{(\alpha)}\right)=\dim_H\left(G_\infty^{(\alpha)}\right)=\frac{1}{1+\theta}.\]
\end{thm}

Our final main result concerns the following situation. Fix a sequence $(s_n)_{n\in\N}$ of natural numbers with the property that $\lim_{n\to\infty}s_n=\infty$. Then, let $\sigma$ be given by
\[
\sigma:=\liminf_{n\to\infty}\frac{\log(s_1\ldots s_n)}{(1+\theta)\log(s_1\ldots s_n)+\theta \log(s_{n+1})}=\frac{1}{(1+\theta)+\theta\left(\limsup\limits_{n\to\infty}\frac{\log(s_{n+1})}{\log(s_1\ldots s_n)}\right)}.
\]
Finally, let $N>3$ and define the set
\[
J_\sigma^{(\alpha)}:=\{x=[\ell_1(x), \ell_2(x), \ldots]_\alpha:s_n\leq \ell_n(x)<Ns_n\text{ for all }n\in\N\}.
\]
We will prove the following theorem.

\begin{thm}\label{jarlu}
\[\dim_H\left(J_\sigma^{(\alpha)}\right)=\sigma.\]
\end{thm}

\begin{rem}
A similar situation for continued fractions has been considered by Fan \textit{et al.}  in \cite{chinese}.
\end{rem}

The proofs of both Theorem \ref{origgood} and  Theorem \ref{goodtwo} will be given in Section 2, while Theorem \ref{jarlu} will be proved in Section 3.

 For future reference, let us  now define the cylinder sets associated with the map $L_\alpha$. For each $k$-tuple $(\ell_1, \ldots, \ell_k)$ of positive integers, define the \index{cylinder set!$\alpha$-L\"uroth}\textit{$\alpha$-L\"uroth cylinder set} $C_\alpha (\ell_1, \ldots, \ell_k)$ associated with  the $\alpha$-L\"{u}roth expansion
 to be
\[
C_\alpha (\ell_1, \ldots, \ell_k):=\{[ y_1, y_2, \ldots
]_\alpha:y_i=\ell_i\text{ for } 1\leq i\leq k\}.
\]
It is easy to see that these cylinder sets are closed intervals with endpoints
given by $[\ell_1, \ldots, \ell_k]_\alpha$ and
$[\ell_1,
\ldots, (\ell_k +1)]_\alpha$. If $k$ is even, we have that $[\ell_1, \ldots, \ell_k]_\alpha$ is the left endpoint of this interval. Likewise, if $k$ is odd, $[\ell_1, \ldots, \ell_k]_\alpha$ is the right endpoint. Directly from the values of its endpoints, for the Lebesgue
measure   $\lambda$ of $C_\alpha
(\ell_1, \ldots,
\ell_k)$ we have that
\[\lambda(C_\alpha (\ell_1, \ldots, \ell_k))=
\prod_{i=1}^{k}a_{\ell_i}.\]

It is assumed that the reader is familiar with the definition and basic properties of the Hausdorff dimension of a set in $\R^n$, which we will denote by $\dim_H$. A good reference on the subject is Falconer's book \cite{Fal}. In particular, we will repeatedly use Frostman's Lemma (also known as the mass distribution principle), which can be found as Theorem 4.2 in \cite{Fal}.

\section{Good-type sets}

We begin this section by proving Theorem \ref{origgood}.

\noindent\textit{Proof of Theorem \ref{origgood}.}
By assumption, $\alpha$ is expansive of exponent $\theta\geq0$. Therefore, from Proposition \ref{mdt}, we have that
$a_n\sim \psi(n)\cdot n^{-(1+\theta)}$, where $\psi:\N\to \R^+$ is a slowly-varying function. This implies that $a_n\asymp \psi(n)\cdot n^{-(1+\theta)}$. Since $\psi$ is slowly varying, it follows that for all positive $\varepsilon$ if $n\in\N$ is sufficiently large, we have that $n^{-\var}\leq \psi(n)\leq n^{\var}$.  Thus, on combining these observations, we obtain that
\[
n^{-(1+\theta +\var)}\leq a_n\leq n^{-(1+\theta -\var)}.
\]

Let $\var>0$ be given. Then, recalling from the introduction that $\lambda(C_\alpha(\ell_1, \ldots, \ell_k))=a_{\ell_1}\ldots a_{\ell_k}$, there exists a positive integer $N:=N(\var)$ such that for each $\alpha$-L\"uroth cylinder set $C_\alpha(\ell_1, \ldots, \ell_k)$ with  $\ell_i>N$ for each $1\leq i\leq k$, we have
\begin{eqnarray}\label{cabove}
\frac{1}{(\ell_1 \ldots \ell_k)^{1+\theta+\var}}\leq\lambda(C_\alpha(\ell_1, \ldots, \ell_k))\leq \frac{1}{(\ell_1 \ldots \ell_k)^{1+\theta-\var}}.
\end{eqnarray}

In order to compute the upper bound, let $\delta>0$ and choose $k$ large enough that
\[
\mathcal{C}:=\{C_\alpha(\ell_1, \ldots, \ell_k):\ell_i>N \text{ for } 1\leq i\leq k\}
\]
is a $\delta$-cover of $G_N^{(\alpha)}$. Let $s:=(1+\theta-\var)^{-1}(1+\var_N)$, where $\var_N$ is chosen to satisfy the conditions that $\var_N<1$ and $-\var_N/\log(\var_N)>1/\log N$. Then,
\begin{eqnarray*}
\mathcal{H}_\delta^s\left(G_N^{(\alpha)}\right)&\leq&\sum_{\mathcal{C}}\lambda(C_\alpha(\ell_1, \ldots, \ell_k))^s\leq\sum_{\mathcal{C}}\left(\left(\frac{1}{\ell_1\ldots\ell_k}\right)^{1+\theta-\var}\right)^{(1+\theta-\var)^{-1}(1+\var_N)}\\
&=&\sum_{\mathcal{C}}\left(\frac{1}{\ell_1\ldots\ell_k}\right)^{1+\var_N}=\left(\sum_{i>N}\left(\frac1{i}\right)^{1+\var_N}\right)^k< \left(\int_{N}^\infty x^{-(1+\var_N)}\ dx\right)^k\\
&=&\left(\frac{1}{\var_N N^{\var_N}}\right)^k<1.
\end{eqnarray*}
Thus, as this estimate is independent of $\delta$, we have that $\dim_H\left(G_N^{(\alpha)}\right)\leq s$. Letting $\var>0$ tend to zero and choosing the sequence $(\var_N)_{N\in\N}$ in such a way that $\lim\limits_{N\to \infty}\var_N=0$, we obtain that
\[
\dim_H\left(G_N^{(\alpha)}\right)\leq \frac{1}{1+\theta}.
\]

In order to calculate the desired lower bound, we define a certain subset of the set $G_N^{(\alpha)}$ for each $N\in\N$. First, choose $M\in\N$ to be such that $\sum_{i=N}^M 1/i>1$. 
Denote this sum by $S$. Then define the set
\[
G_{N,M}^{(\alpha)}:=\{x=[\ell_1(x), \ell_2(x), \ldots]_\alpha\in\U:N\leq\ell_i(x)\leq M\text{ for all }i\in\N\}.
\]
Clearly, $G_{N,M}^{(\alpha)}\subseteq G_{N}^{(\alpha)}$ and so a lower bound for the Hausdorff dimension of the subset $G_{N,M}^{(\alpha)}$ is also a lower bound for the set $G_{N}^{(\alpha)}$.
We aim to use Frostman's Lemma, so, to that end, define a mass distribution $\nu$ on the set $G_{N,M}^{(\alpha)}$ by setting
\begin{eqnarray*}
\nu(C_\alpha(\ell_1, \ldots, \ell_k))&:=&\frac{1}{S^k\ell_1\ldots \ell_k}.
\end{eqnarray*}
Note that from (\ref{cabove}), if $N$ is large enough, we have that
\[
\nu(C_\alpha(\ell_1, \ldots, \ell_k))\leq \left(\frac{1}{S}\right)^k\lambda(C_\alpha(\ell_1, \ldots, \ell_k))^{1/(1+\theta+\var)}< \lambda(C_\alpha(\ell_1, \ldots, \ell_k))^{1/(1+\theta+\var)},
\]
where the second inequality comes from the fact that $1/S<1$.
Also note that \[\frac{\lambda(C_\alpha(\ell_1, \ldots, \ell_k))}{\lambda(C_\alpha(\ell_1, \ldots, \ell_k,\ell_{k+1}))}=\frac1{a_{k+1}}\leq \ell_{k+1}^{1+\theta+\var}\leq  M^{1+\theta+\var}.\]

Now, let $x=[\ell_1(x), \ell_2(x), \ldots]_\alpha\in G_{N,M}^{(\alpha)}$, let $r>0$ and further let $k\in\N$ be such that  we have \[\lambda(C_\alpha(\ell_1(x), \ldots, \ell_{k+1}(x)))\leq r<\lambda(C_\alpha(\ell_1(x), \ldots, \ell_{k}(x))).\] It is clear that $C_\alpha(\ell_1(x), \ldots, \ell_k(x), \ell_{k+1}(x))\subset B(x, r)$, but it is possible that $B(x, r)$ intersects more than one cylinder set of length $k$. However, since there are at most $M-N$ possibilities and the $\nu$-measure of each of them is comparable, without loss of generality we can assume that
\[
C_\alpha(\ell_1(x), \ldots, \ell_{k+1}(x))\subset B(x, r)\subset C_\alpha(\ell_1(x), \ldots, \ell_{k}(x)).
\]
Then,
\begin{eqnarray*}
\nu(B(x, r))&\leq& \nu(C_\alpha(\ell_1(x), \ldots, \ell_{k}(x)))\leq \lambda(C_\alpha(\ell_1(x), \ldots, \ell_{k}(x)))^{1/(1+\theta+\var)}\\&\leq&
M\lambda(C_\alpha(\ell_1(x), \ldots, \ell_{k+1}(x)))^{1/(1+\theta+\var)}\leq M r^{1/(1+\theta+\var)}.
\end{eqnarray*}
Hence, by Frostman's Lemma, it follows that
\[
\dim_H\left(G_{N,M}^{(\alpha)}\right)\geq \frac{1}{1+\theta+\var}.
\]
Finally, on letting  $\var$ tend to zero, we have that
\[
\lim_{N\to \infty}\dim_H\left(G_{N,M}^{(\alpha)}\right)\geq \frac{1}{1+\theta}.
\]
Combining this with the upper bound given above finishes the proof of the theorem.

\bewend

Let us now move on to the proof of Theorem \ref{goodtwo}.
The proof will again be split into the lower bound and the upper bound. We begin  with the following useful lemma.

\begin{lem}\label{int}
Suppose that $x=[\ell_1(x), \ell_2(x), \ldots]_\alpha\in F_{\infty}^{(\alpha)}$. Further suppose that
\[
\lambda(C_\alpha(\ell_1(x), \ldots,\ell_k(x), \ell_{k+1}(x) ))\leq r<\lambda(C_\alpha(\ell_1(x), \ldots,\ell_k(x))).
\]
Then, for all sufficiently large $k$,
\[
B(x, r)\subset \bigcup_{i=-1}^{1}C_\alpha(\ell_1(x), \ldots,\ell_k(x)+i ).
\]
\end{lem}

\begin{proof}
We consider here only the case in which $k$ is odd, the case $k$ even is analogous and is left to the reader. Bearing in mind that $x\in C_\alpha(\ell_1(x), \ldots, \ell_{k+1}(x))$, it is clear that if $k$ is sufficiently large, then the right endpoint of $B(x, r)$ cannot extend past the interval $C_\alpha(\ell_1(x), \ldots, \ell_{k}(x)-1)$, as we are assuming that the partition $\alpha$ is eventually decreasing. 
On the other hand, the left endpoint of $B(x, r)$ cannot be smaller than the point $[\ell_1(x), \ldots, \ell_{k+1}(x)]_\alpha- a_{\ell_1(x)}\ldots a_{\ell_k(x)}$. But this point is equal to
\[
\left(t_{\ell_1(x)}-a_{\ell_1(x)}t_{\ell_2(x)}+\ldots+a_{\ell_1(x)}\ldots a_{\ell_{k-1}(x)}t_{\ell_k(x)}-a_{\ell_1(x)}\ldots a_{\ell_k(x)}t_{\ell_{k+1}(x)}\right)-a_{\ell_1(x)}\ldots a_{\ell_k(x)}\]
\[=t_{\ell_1(x)}-\ldots +a_{\ell_1(x)}\ldots a_{\ell_{k-1}(x)}(t_{\ell_k(x)}-a_{\ell_k(x)})-a_{\ell_1(x)}\ldots a_{\ell_k(x)}t_{\ell_{k+1}(x)}\hspace{3.2cm}\]
\[= [\ell_1(x), \ldots, \ell_{k-1}(x), \ell_{k}(x)+1]_\alpha-a_{\ell_1(x)}\ldots a_{\ell_k(x)}t_{\ell_{k+1}(x)}.\hspace{4.85cm}
\]

\vspace{1mm}

\noindent Notice that the point $[\ell_1(x), \ldots, \ell_{k-1}(x), \ell_{k}(x)+1]_\alpha$ is the left endpoint of the the cylinder set $C_\alpha(\ell_1(x), \ldots, \ell_k(x))$, so it only remains to prove that
\[
a_{\ell_1(x)}\ldots a_{\ell_{k-1}(x)}a_{\ell_k(x)}t_{\ell_{k+1}(x)}\leq a_{\ell_1(x)}\ldots a_{\ell_{k-1}(x)}a_{\ell_k(x)+1}=\lambda(C_\alpha(\ell_1(x), \ldots, \ell_k(x)+1)).
\]
In other words, we must show that
\[
a_{\ell_k(x)}t_{\ell_{k+1}(x)}\leq a_{\ell_k(x)+1}.
\]
Recall that $\alpha$ is assumed to be expanding of exponent $\theta\geq0$, so $t_n=n^{-\theta}\cdot \psi(n)$ and $a_n\asymp n^{-(1+\theta)}\cdot \psi(n)$, where $\psi:\N\to\R^+$ is a slowly varying function. Also recall that for each positive $\var$, if $n$ is sufficiently large, we have that  $t_n\leq n^{-(\theta-\var)}$ and $n^{-(1+\theta+\var)}\leq a_n\leq n^{-(1+\theta-\var)}$. Let $\var<\theta/6$. Then, since $x\in F_{\infty}^{(\alpha)}$, so $\ell_k(x)\leq\ell_{k+1}(x)$ for all $k$, we have that
\begin{eqnarray*}
a_{\ell_k(x)}t_{\ell_{k+1}(x)}&\leq& \frac{1}{\ell_k(x)^{(1+\theta-\var)}}\cdot\frac{1}{\ell_{k+1}(x)^{(\theta-\var)}}\\
&\leq&\frac{1}{\ell_k(x)^{(1+2\theta-2\var)}}<\frac{1}{\ell_k(x)^{(1+(5/3)\theta)}}.
\end{eqnarray*}
On the other hand, we also have that
\[
a_{\ell_k(x)+1}\geq \frac{1}{(\ell_k(x)+1)^{(1+\theta+\var)}}>\frac{1}{(\ell_k(x)+1)^{(1+(7/6)\theta)}}.
\]
Therefore, in order to show that $a_{\ell_k(x)}t_{\ell_{k+1}(x)}\leq a_{\ell_k(x)+1}$, it suffices to show that
\[
\frac{1}{\ell_k(x)^{(1+(5/3)\theta)}}<\frac{1}{(\ell_k(x)+1)^{(1+(7/6)\theta)}},
\]
or, equivalently, that
\[
1-\frac{(1/2)\theta}{1+(3/5)\theta}=\frac{1+(7/6)\theta}{1+(10/6)\theta}<\frac{\log(\ell_k(x))}{\log(\ell_k(x)+1)}.
\]
But, since the left-hand side is a fixed amount less than 1, depending only on $\theta$, and the right-hand side tends to 1 as $\ell_k(x)$ increases (that is, as $k$ increases), it follows that if $k$ is large enough, this statement is true. Thus, the left endpoint of $B(x, r)$ lies in $ C_\alpha(\ell_1(x), \ldots,\ell_k(x)+1 )$ and the lemma is proved.

\end{proof}

In the next lemma, we will establish the lower bound for the dimension of $F_\infty^{(\alpha)}$.

\begin{lem}\label{lowboundGinfty}
\[
\frac{1}{1+\theta}\leq \dim_H\left(F_\infty^{(\alpha)}\right).
\]

\end{lem}

\begin{proof}
 For the proof, we will define a suitable subset of $F_\infty^{(\alpha)}$ and use Frostman's Lemma again to obtain the lower bound. So, first let $f_\var:\N\to\N$ be a slowly varying function which satisfies the following properties:
  \begin{itemize}
    \item $\lim_{n\to\infty}f_\var(n)=\infty$.
    \item $f_\var(n)\leq f_\var(n+1)$ for all $n\in\N$.
    \item $f_\var(1)$ is large enough that if $\ell\geq f_\var(1)$, then $a_\ell\geq \ell^{-(1+\theta+\var)}$.
  \end{itemize}
  Now,  define a second function $g:\N\to\N$ by setting $g(n)$ to be the least integer such that
\[
S_n:=\sum_{i=f_\var(n)}^{g(n)}\frac1i>1.
\]
Note that the function $g$ is also slowly varying. Indeed, for any $k\in\N$, if $f_\var(n)\in\{2^k+1, \ldots, 2^{k+1}\}$ it follows that $2^{k+1}\leq g(n)\leq 2^{k+3}$. Hence, $f_\var(n)<g(n)\leq 8 f_\var(n)$.
Finally,  define the set
\[
F_{f_\var, g}^{(\alpha)}:=\left\{x=[\ell_1(x), \ell_2(x), \ldots]_\alpha: f_\var(n)\leq \ell_n(x)\leq g(n) \text{ and } \ell_n(x)\geq \ell_{n-1}(x)\text{ for all }n\in\N\right\}.
\]

It is clear that $F_{f_\var,g}^{(\alpha)}\subset F_\infty^{(\alpha)}$.
Define a mass distribution on $F_{f_\var,g}^{(\alpha)}$ by setting
\[
\nu(C_\alpha(\ell_1(x), \ldots, \ell_k(x))):=\frac{1}{S_1\ldots S_k}\cdot\frac{1}{\ell_1(x) \ldots \ell_k(x)}.
\]
Note that due to the choice of $f_\var$ and $g$, we have that
\[
\nu(C_\alpha(\ell_1(x), \ldots, \ell_k(x)))\leq \lambda(C_\alpha(\ell_1(x), \ldots, \ell_k(x)))^{1/(1+\theta+\var)}.
\]
In addition, observe that
\[
\frac{\lambda(C_\alpha(\ell_1(x), \ldots, \ell_k(x)))}{\lambda(C_\alpha(\ell_1(x), \ldots, \ell_{k+1}(x)))}=\frac{1}{a_{\ell_{k+1}(x)}}\leq \ell_{k+1}(x)^{1+\theta+\var}\leq g(k+1)^{1+\theta+\var}.
\]

As in the proof of Theorem \ref{origgood}, let $r>0$ and choose $k$ such that \[\lambda(C_\alpha(\ell_1(x), \ldots, \ell_{k+1}(x)))\leq r<\lambda(C_\alpha(\ell_1(x), \ldots, \ell_{k}(x))).\] Again, it is clear that $C_\alpha(\ell_1(x), \ldots, \ell_k(x), \ell_{k+1}(x))\subset B(x, r)$, but it is possible that $B(x, r)$ intersects more than one interval in level $k$. There are no longer a fixed finite set of possibilities, but for large enough $k$ (that is, for small enough $r$), we can apply Lemma \ref{int} to conclude that
\[
C_\alpha(\ell_1(x), \ldots, \ell_{k+1}(x))\subset B(x, r)\subset \bigcup_{i=-1}^{1}C_\alpha(\ell_1(x), \ldots,\ell_k(x)+i ).
\]

Now, let $\delta>0$ be arbitrary. Then, recall that $g$ is slowly varying, so that if $k$ is large enough, we have that $g(k+1)\leq(k+1)^\delta<(1/r)^\delta$. Then, the proof of the lemma follows from the following calculation.
\begin{eqnarray*}
\nu(B(x, r))&\ll& \nu(C_\alpha(\ell_1(x), \ldots,\ell_k(x)))\leq \lambda(C_\alpha(\ell_1(x), \ldots,\ell_k(x)))^{1/(1+\theta+\var)}\\
&\leq& g(k+1) \lambda(C_\alpha(\ell_1(x), \ldots,\ell_k(x), \ell_{k+1}(x)))^{1/(1+\theta+\var)}\\
&\ll&g(k+1)\cdot r^{1/(1+\theta+\var)}\\&\leq& r^{1/(1+\theta+\var)-\delta}.
\end{eqnarray*}
Since this is true for all $\delta>0$, an application of Frostman's Lemma yields that
\begin{eqnarray}\label{varsmall}
\frac{1}{1+\theta+\var}\leq \dim_{H}\left(F_{f_\var, g}^{(\alpha)}\right).
\end{eqnarray}
Finally, (\ref{varsmall}) shows that for every $\var>0$ we have that $\dim_H\left(F_{\infty}^{(\alpha)}\right)\geq1/(1+\theta+\var) $, so letting $\var$ tend to zero completes the proof.

\end{proof}

All that remains for the proof of Theorem \ref{goodtwo} is to give the upper bound for the dimension of $G_\infty^{(\alpha)}$. For this, first observe that if we consider the set
\[
G_{N, n_0}^{(\alpha)}:=\left\{x=[\ell_1(x), \ell_2(x), \ldots]_\alpha:\ell_{n}(x)>N\text{ for all }n\geq n_0\right\},
\]
we can easily see that for all $n_0\in\N$ this set has the same dimension as the set $G_N^{(\alpha)}$. It is also clear that for all $N\in\N$ there exists an $n_0$  such that $G_\infty^{(\alpha)}\subset G_{N, n_0}^{(\alpha)}$. Therefore, it follows from Theorem \ref{origgood} that
\[
\dim_{H}\left(G_\infty^{(\alpha)}\right)\leq \frac{1}{1+\theta}.
\]
Taking this observation together with Lemma \ref{lowboundGinfty}, we have proved Theorem \ref{goodtwo}.

\section{Strict Jarn\'{\i}k sets}

In this section, we give the proof of Theorem \ref{jarlu}. Before beginning, notice that for each $\sigma\in\R^+$ the set $J_\sigma^{(\alpha)}$ is contained in the set $G^{(\alpha)}_\infty$. Therefore the dimension can be at most $1/(1+\theta)$. This is consistent with the result given here, since, as we recall from the introduction, we have that $\sigma=1/((1+\theta)+\theta\cdot \tau)$, where $\tau:=\limsup_{n\to\infty}\log(s_{n+1})/\log(s_1\ldots s_n)\geq0$.

\noindent\textit{Proof of Theorem \ref{jarlu}.}
Let us begin by establishing the upper bound. The set $J_\sigma^{(\alpha)}$ can be covered by sets of the form
\[
\widetilde{C}_\alpha(\ell_1, \ldots, \ell_k):=\bigcup_{m\geq s_{k+1}}C_\alpha(\ell_1, \ldots, \ell_k, m),
\]
where $s_i\leq \ell_i< Ns_i$ for each $1\leq i\leq k$. We have that
\[
\lambda(\widetilde{C}_\alpha(\ell_1, \ldots, \ell_k))=a_{\ell_1}\ldots a_{\ell_k}t_{s_{k+1}}.
\]
Recall that since $\alpha$ is expansive of exponent $\theta$ and eventually decreasing, for each positive $\var$, there exists $k\in\N$ such that  $\ell^{-(1+\theta+\var)}\leq a_\ell\leq \ell^{-(1+\theta-\var)}$ for all $\ell\geq k$. Since the sequence $(s_n)_{n\in\N}$ tends to infinity, we may assume without loss of generality that if $x\in J_{\sigma}^{(\alpha)}$, then $(\ell_n(x))^{-(1+\theta+\var)}\leq a_{\ell(x)}\leq (\ell_n(x))^{-(1+\theta-\var)}$ for all $n\in\N$. For each $x\in J_{\sigma}^{(\alpha)}$, these observations lead to the estimate
\[
\frac{1}{(\ell_1(x)\ldots \ell_k(x))^{(1+\theta+\var)}(s_{k+1})^{(\theta+\var)}}\leq\lambda(\widetilde{C}_\alpha(\ell_1(x), \ldots, \ell_k(x)))\leq\frac{1}{(\ell_1(x)\ldots \ell_k(x))^{(1+\theta-\var)}(s_{k+1})^{(\theta-\var)}}.
\]
In turn, this yields
\begin{eqnarray}\label{L3.1}
\ \ \ \ \ \frac{1}{(N^k s_1\ldots s_k)^{(1+\theta+\var)}(s_{k+1})^{(\theta+\var)}}\leq\lambda(\widetilde{C}_\alpha(\ell_1(x), \ldots, \ell_k(x)))\leq\frac{1}{(s_1\ldots s_k)^{(1+\theta-\var)}(s_{k+1})^{(\theta-\var)}}.
\end{eqnarray}

Now, define
\[
\sigma_\var:=\liminf_{n\to\infty}\frac{\log(s_1\ldots s_n)}{(1+\theta-\var)\log(s_1\ldots s_n)+(\theta-\var) \log(s_{n+1})}.
\]

Directly from this definition, we have that if $\sigma'\in(\sigma_\var, 3\sigma_\var)$ and $n$ is sufficiently large, then
\[
\frac{\sigma'-\sigma_\var}{2}\leq \frac{\log(s_1\ldots s_n)}{\log\left((s_1\ldots s_n)^{(1+\theta-\var)}(s_{n+1})^{(\theta-\var)}\right)}.
\]
Thus,
\[
\left(\frac{1}{(s_1\ldots s_n)^{(1+\theta-\var)}(s_{n+1})^{\theta-\var}}\right)^{\frac{\sigma'-\sigma_\var}{2}}\leq\left(\frac{1}{(s_1\ldots s_n)^{(1+\theta-\var)}(s_{n+1})^{\theta-\var}}\right)^{\frac{\log(s_1\ldots s_n)}
{\log(s_1\ldots s_n)^{(1+\theta-\var)}(s_{n+1})^{(\theta-\var)}}}=\frac1{s_1\ldots s_n}.
\]
It follows that $s_1\ldots s_n\leq \left((s_1\ldots s_n)^{(1+\theta-\var)}(s_{n+1})^{\theta-\var}\right)^{\frac{\sigma'-\sigma_\var}{2}}$. Now, since
 $\lim_{n\to\infty}s_n=\infty$, we immediately have that $\lim_{n\to\infty}\log(s_n)=\infty$ and this in turn implies that $\lim_{n\to\infty}(\log(s_1\ldots s_n))/n=\infty$. Therefore, for large enough $n\in\N$ we have that $\log(N-1)<\log(s_1\ldots s_n)/n$. From this, we obtain that
\begin{eqnarray}\label{L3.2}
(N-1)^n\leq \left((s_1\ldots s_n)^{(1+\theta-\var)}(s_{n+1})^{\theta-\var}\right)^{\frac{\sigma'-\sigma_\var}{2}}.
\end{eqnarray}

Again from the definition of $\sigma_\var$, there exists a sequence $(n_k)_{k\in\N}$ of positive integers such that if $\sigma'>\sigma_\var$, we have
\[
\frac{\log(s_1\ldots s_{n_k})}{\log\left((s_1\ldots s_{n_k})^{(1+\theta-\var)}(s_{n_k+1})^{(\theta-\var)}\right)}\leq \frac{\sigma'+\sigma_\var}{2}.
\]
Thus,
\begin{eqnarray}\label{L3.3}
s_1\ldots s_{n_k}\leq \left((s_1\ldots s_{n_k})^{(1+\theta-\var)}(s_{n_k+1})^{\theta-\var}\right)^{\frac{\sigma'+\sigma_\var}{2}}.
\end{eqnarray}
Consequently, if we neglect any terms of the sequence $(n_k)$ that are too small and rename the sequence accordingly, by combining the estimates in (\ref{L3.2}) and (\ref{L3.3}), we obtain for all $k\geq1$ that
\[
(N-1)^{n_k}s_1\ldots s_{n_k}\leq \left((s_1\ldots s_{n_k})^{(1+\theta-\var)}(s_{n_k+1})^{\theta-\var}\right)^{\sigma'}.
\]

Thus,
\begin{eqnarray*}
\mathcal{H}^{\sigma'}(J_\sigma^{(\alpha)})&=&\liminf_{k\to\infty}\sum_{(\ell_1, \ldots, \ell_{n_k})\atop s_i\leq\ell_i<Ns_i}\lambda(\widetilde{C}_\alpha(\ell_1, \ldots, \ell_k))^{\sigma'}\\
&\leq& (N-1)^{n_k}s_1\ldots s_{n_k}\cdot\left((s_1\ldots s_{n_k})^{-(1+\theta-\var)}(s_{n_k+1})^{-(\theta-\var)}\right)^{\sigma'}\leq1.
\end{eqnarray*}

Hence, for all $\var>0$ and all $\sigma'>\sigma_\var$, we have that $\dim_H(J_\sigma^{(\alpha)})\leq \sigma'$ and so,  $\dim_H(J_\sigma^{(\alpha)})\leq \sigma_\var$. It therefore follows, on letting $\var$ tend to zero, that
\[
 \dim_H(J_\sigma^{(\alpha)})\leq \sigma.
\]

Let us now provide the lower bound. For this, as usual, we will use Frostman's Lemma. To that end, define a mass distribution $m$ on $J_\sigma^{(\alpha)}$ by setting $m(C_\alpha(\ell_1, \ldots, \ell_k))=1/(\ell_1\ldots\ell_k)$. Let $x\in J_\sigma^{(\alpha)}$, $r>0$ and choose $k$ such that
\[
\lambda(\widetilde{C}_\alpha(\ell_1(x), \ldots, \ell_{k+1}(x)))\leq r <\lambda(\widetilde{C}_\alpha(\ell_1(x), \ldots, \ell_k(x))).
\]
There are now two possibilities. Either,
\begin{eqnarray}\label{L3.4}
\lambda(\widetilde{C}_\alpha(\ell_1(x), \ldots, \ell_{k+1}(x)))\leq r <\lambda({C}_\alpha(\ell_1(x), \ldots, \ell_k(x), \ell_{k+1}(x))),
\end{eqnarray}
or,
\begin{eqnarray}\label{L3.5}
\lambda({C}_\alpha(\ell_1(x), \ldots, \ell_{k+1}(x)))\leq r <\lambda(\widetilde{C}_\alpha(\ell_1(x), \ldots, \ell_k(x))).
\end{eqnarray}

Suppose we are in the situation of (\ref{L3.4}) and, for simplicity, assume that $k$ is odd. It is clear that if $k$  is large enough, the left endpoint of the ball $B(x, r)$ cannot extend past the cylinder set $C_{\alpha}(\ell_1(x), \ldots, \ell_{k+1}(x)-1)$ (since $\alpha$ is assumed to be eventually decreasing). On the other hand, the right endpoint cannot be larger than $[\ell_1(x), \ldots, \ell_{k+1}, 1]_\alpha+ a_{\ell_1(x)}\ldots a_{\ell_{k+1}(x)}t_{s_{k+2}}$. We claim that as long as $k$ is chosen large enough, this point lies inside $C_{\alpha}(\ell_1(x), \ldots, \ell_{k+1}(x)+1)$. To prove this claim, we are required to show that
\[
a_{\ell_1(x)}\ldots a_{\ell_{k+1}(x)}t_{s_{k+2}}<a_{\ell_1(x)}\ldots a_{\ell_{k+1}(x)+1},
\]
or, in other words, that
\[
a_{\ell_{k+1}(x)}t_{s_{k+2}}<a_{\ell_{k+1}(x)+1}.
\]
Note that by choosing $k$ sufficiently large, the value of $t_{s_{k+2}}$ can be made as small as we like, so it is enough to show that there exists some constant $K$ with the property that for all large enough $n\in\N$,
\[
\frac{a_{n}}{a_{n+1}}\leq K.
\]
Since $\alpha$ is expansive of exponent $\theta$, we have that
\[
\frac{a_{n}}{a_{n+1}}\leq  \frac{c(n+1)^{1+\theta}\psi(n)}{n^{1+\theta}\psi(n+1)}.
\]
It is obvious that $\lim_{n\to\infty}((n+1)/n)^{1+\theta}=1$, so all that remains to establish the claim is to show that $\lim_{n\to\infty}{\psi(n)}/{\psi(n+1)}\leq1$. In order to do this, suppose by way of contradiction that \[\lim_{n\to\infty}{\psi(n)}/{\psi(n+1)}>1.\] Then, recalling that $\psi$ is a slowly-varying function, we have that $\lim_{n\to\infty}{\psi(cn)}/{\psi(n)}=1$ for all $c>0$. Therefore,  we obtain that
\[
\lim_{n\to\infty}\left(\frac{\psi(n)}{\psi(2n)}\cdot\frac{\psi(2n)}{\psi(n+1)}\right)=\lim_{n\to\infty}\frac{\psi(2n)}{\psi(n+1)}>1.
\]
Thus, there exists $n_0\in\N$ such that for all $n\geq n_0$, we have that
\[
\frac{\psi(n)}{\psi(n+1)}>1\ \text{ and }\ \frac{\psi(2n)}{\psi(n+1)}>1.
\]
This implies that for $n\geq n_0$ we have
\[
\psi(n)>\psi(n+1)>\psi(n+2)>\ldots> \psi(2n-1)>\psi(2n)>\psi(n+1).
\]
This contradiction finishes the proof.

In a slight abuse of notation, let us redefine the quantity $\sigma_\var$ used above in the following way:
\[
\sigma_\var:=\liminf_{n\to\infty}\frac{\log(s_1\ldots s_n)}{(1+\theta+\var)\log(s_1\ldots s_n)+(\theta+\var) \log(s_{n+1})}.
\]

We have shown that $B(x, r)\subset \bigcup_{i=-1}^1C_{\alpha}(\ell_1(x), \ldots, \ell_{k+1}(x)+i)$. 
Therefore, if we let $\sigma'<\sigma_\var$ and bear in mind that $r\geq a_{\ell_1(x)}\ldots a_{\ell_{k+1}(x)}t_{s_{k+2}}$, we obtain, via (\ref{L3.1}) and the definition of $\sigma_\var$, that
\begin{eqnarray*}
m(B(x, r))&\leq& 3m(C_{\alpha}(\ell_1(x), \ldots, \ell_{k+1}(x)))\leq\frac3{s_1\ldots s_{k+1}}\\
&\leq&3\left(\frac{1}{(s_1\ldots s_{k+1})^{(1+\theta+\var)}(s_{k+2})^{(\theta+\var)}}\right)^{\sigma'}\\
&\leq& 3r^{\sigma'}.
\end{eqnarray*}

In this case, then, an application of Frostman's Lemma yields that for all $\var>0$ and all $\sigma'<\sigma_\var$, we have that
\[
\dim_{H}{\left(J_\sigma^{(\alpha)}\right)}\geq \sigma'.
\]

Let us now consider the second case, that of (\ref{L3.5}). Again, suppose for the sake of argument that $k$ is odd. Then, it is clear once more that if $k$ is large enough, the right endpoint of $B(x, r)$ cannot extend past the cylinder set $C_\alpha(\ell_1(x), \ldots, \ell_k(x)-1)$, since
$\alpha$ is eventually decreasing. On the other hand, the left endpoint of $B(x, r)$ is not less than $[\ell_1(x), \ldots, \ell_k(x)]_\alpha-2a_{\ell_1(x)}\ldots a_{\ell_k(x)}t_{s_{k+1}}$. If $k$ is sufficiently large, it is clear that $2a_{\ell_1(x)}\ldots a_{\ell_k(x)}t_{s_{k+1}}<a_{\ell_1(x)}\ldots a_{\ell_k(x)}$ (as $t_{s_{k+1}}$ can be made arbitrarily small by choosing large enough $k$). This implies that the left endpoint of $B(x, r)$ is contained within the cylinder set $C_\alpha(\ell_1(x), \ldots, \ell_k(x))$ and consequently $B(x, r)$ can only intersect the sets $\widetilde{C}_\alpha(\ell_1(x), \ldots, \ell_k(x))$ and $\widetilde{C}_\alpha(\ell_1(x), \ldots, \ell_k(x)-1)$ in this level.

Also, note that the smallest size of a cylinder set in the $(k+1)$-th level is $(N^{k+1}s_1\ldots s_{k+1})^{-(1+\theta+\var)}$. Consequently, at most $2r(N^{k+1}s_1\ldots s_{k+1})^{(1+\theta+\var)}$ of these cylinder sets can intersect $B(x, r)$. Taking these observations together, we have that
\begin{eqnarray*}
m(B(x, r))&\leq& \min\left\{2m(\widetilde{C}_\alpha(\ell_1(x), \ldots, \ell_k(x))), \left(2r(N^{k+1}s_1\ldots s_{k+1})^{(1+\theta+\var)}\right)m\left(C_\alpha(\ell_1(x), \ldots, \ell_{k+1}(x))\right)\right\}\\
&\leq&\min \left\{ \frac{2}{s_1\ldots s_k}, \frac{2(N^{k+1}s_1\ldots s_{k+1})^{(1+\theta+\var)}\cdot r}{s_1\ldots s_k s_{k+1}}\right\}\\
&=&\frac{2}{s_1\ldots s_k}\left\{1, \left((N^{k+1}s_1\ldots s_{k})^{(1+\theta+\var)}(s_{k+1})^{(\theta+\var)}\right)\cdot r\right\}.
\end{eqnarray*}
Note that $\min\{a, b\}\leq a^{1-s}b^s$ for all $s\in(0,1)$ and let $\sigma'<\sigma_\var$. It follows from this that
\[
m(B(x, r))\leq \frac{2}{s_1\ldots s_k}\left((N^{k+1}s_1\ldots s_{k})^{(1+\theta+\var)}(s_{k+1})^{(\theta+\var)}\right)^{\sigma'}r^{\sigma'}.
\]
By definition of $\sigma_\var$, we have for all $\sigma'<\sigma_\var$ and all large enough $k$ that
\[
\frac{1}{s_1\ldots s_k}\leq\left((N^{k+1}s_1\ldots s_{k})^{(1+\theta+\var)}(s_{k+1})^{(\theta+\var)}\right)^{\sigma'}.
\]
Thus,
\[
m(B(x, r))\leq 2 r^{\sigma'}.
\]
Therefore, as in the case of (\ref{L3.4}) described above, for
all $\var>0$ and all $\sigma'<\sigma_\var$, we have that
\[
\dim_{H}{\left(J_\sigma^{(\alpha)}\right)}\geq \sigma'.
\]
Finally, since this holds in both cases for all $\sigma'<\sigma_\var$, we first obtain that $\dim_{H}{\left(J_\sigma^{(\alpha)}\right)}\geq \sigma_\var$ and then, by letting $\var$ tend to zero, we obtain that
\[
\dim_{H}{\left(J_\sigma^{(\alpha)}\right)}\geq \sigma.
\]
Combining this lower bound with the upper bound given above completes the proof of the theorem.

\bewend

\thanks{The author would like to thank the organisers of the Numeration workshop for providing such a stimulating working environment. She would also like to thank her supervisor for many helpful discussions.}

\end{document}